\documentclass[10pt,a4paper,reqno]{amsart}

\usepackage{amssymb}
\usepackage{amsmath}
\usepackage[all]{xy}
\usepackage{latexsym}
\usepackage{amsthm}
\usepackage{a4wide}
\usepackage{color}
\usepackage{hyperref}
\input diagxy

\theoremstyle{plain}
  \newtheorem{thm}{Theorem}[section]
  
  \newtheorem{prop}[thm]{Proposition}
  \newtheorem{cor}[thm]{Corollary}
\theoremstyle{definition}

  \newtheorem{exmp}[thm]{Example}
  \newtheorem{rem}[thm]{Remark}

\DeclareMathOperator*{\colim}{colim}

\DeclareMathOperator{\dom}{dom}
\DeclareMathOperator{\cod}{cod}

\DeclareMathOperator{\ob}{ob}

\renewcommand{\phi}{\varphi}
\newcommand{\da}{\downarrow}

\newcommand{\lda}{\swarrow}
\newcommand{\rda}{\searrow}
\newcommand{\Lra}{\Longrightarrow}

\newcommand{\oto}{{\to\hspace*{-3.1ex}{\circ}\hspace*{1.9ex}}}

\newcommand{\bv}{\bigvee}
\newcommand{\bw}{\bigwedge}

\newcommand{\dv}{\dashv}

\newcommand{\nat}{\natural}

\newcommand{\al}{\alpha}
\newcommand{\be}{\beta}

\newcommand{\ep}{\epsilon}
\newcommand{\de}{\delta}
\newcommand{\De}{\Delta}
\newcommand{\ga}{\gamma}
\newcommand{\Ga}{\Gamma}
\newcommand{\ka}{\kappa}

\newcommand{\lam}{\lambda}

\newcommand{\CA}{\mathcal{A}}

\newcommand{\CC}{\mathcal{C}}

\newcommand{\CJ}{\mathcal{J}}

\newcommand{\CQ}{\mathcal{Q}}

\newcommand{\CX}{\mathcal{X}}

\newcommand{\sD}{{\sf D}}
\newcommand{\sP}{{\sf P}}

\newcommand{\sy}{{\sf y}}

\newcommand{\Cat}{{\bf Cat}}
\newcommand{\CAT}{{\bf CAT}}

\newcommand{\Chu}{{\bf Chu}}

\newcommand{\Dis}{{\bf Dis}}

\newcommand{\Met}{{\bf Met}}

\newcommand{\ParMet}{{\bf ParMet}}
\newcommand{\ParOrd}{{\bf ParOrd}}

\newcommand{\Ord}{{\bf Ord}}

\newcommand{\Set}{{\bf Set}}

\newcommand{\Sup}{{\bf Sup}}
\newcommand{\QCat}{\CQ\text{-}\Cat}
\newcommand{\DQCat}{\DQ\text{-}\Cat}

\newcommand{\QDis}{\CQ\text{-}\Dis}

\newcommand{\QChu}{\CQ\text{-}\Chu}

\newcommand{\co}{{\rm co}}

\newcommand{\op}{{\rm op}}

\newcommand{\of}{\overline{f}}
\newcommand{\og}{\overline{g}}

\newcommand{\tphi}{\widetilde{\phi}}
\newcommand{\tchi}{\widetilde{\chi}}

\newcommand{\tpsi}{\widetilde{\psi}}

\newcommand{\ha}{\widehat{a}}

\newcommand{\hC}{\widehat{C}}

\newcommand{\hX}{\widehat{X}}

\newcommand{\PU}{\sP U}

\newcommand{\PX}{\sP X}
\newcommand{\PY}{\sP Y}

\newcommand{\DQ}{\sD\CQ}

\newcommand{\tkD}{\widetilde{\ka D}}

\numberwithin{equation}{section}

\begin{document}

\title{Limits and colimits of quantaloid-enriched categories and their distributors}

\author{Lili Shen}
\address{Lili Shen\\Department of Mathematics and Statistics, York University, Toronto, Ontario, Canada, M3J 1P3}
\email{shenlili@yorku.ca}

\author{Walter Tholen}
\address{Walter Tholen\\Department of Mathematics and Statistics, York University, Toronto, Ontario, Canada, M3J 1P3}
\email{tholen@mathstat.yorku.ca}

\thanks{Partial financial assistance by the Natural Sciences and Engineering Research Council of Canada is gratefully acknowledged.}
\dedicatory{Dedicated to the memory of Reinhard B{\"o}rger}

\date{}

\begin{abstract}
It is shown that, for a small quantaloid $\CQ$, the category of small $\CQ$-categories and $\CQ$-functors is total and cototal, and so is the category of $\CQ$-distributors and $\CQ$-Chu transforms.
\end{abstract}

\maketitle


\section{Introduction}

The importance of (small) categories enriched in a (unital) quantale rather than in an arbitrary monoidal category was discovered by Lawvere \cite{Lawvere1973} who enabled us to look at individual mathematical objects, such as metric spaces, as small categories. Through the study of lax algebras \cite{Hofmann2014}, quantale-enriched categories have become the backbone of a larger array of objects that may be viewed as individual generalized categories. Prior to this development, Walters \cite{Walters1981} had extended Lawvere's viewpoint in a different manner, replacing the quantale at work by a \emph{quantaloid} (a term proposed later by Rosenthal \cite{Rosenthal1996}), thus by a bicategory with the particular property that its hom-objects are given by complete lattices such that composition from either side preserves suprema; quantales are thus simply one-object quantaloids.

Based on the theory of quantaloid-enriched categories developed by Stubbe \cite{Stubbe2005,Stubbe2006}, recent works \cite{Hohle2011,Pu2012,Stubbe2014,Tao2014} have considered in particular the case when the quantaloid in question arises from a given quantaloid by a ``diagonal construction'' whose roots go far beyond its use in this paper; see \cite{Grandis2000,Grandis2002}. Specifically for the one-object quantaloids (i.e., quantales) whose enriched categories give (pre)ordered sets and (generalized) metric spaces, the corresponding small quantaloids of diagonals lead to truly partial structures, in the sense that the full structure is available only on a subset of the ambient underlying set of objects. 

In the first instance then, this paper aims at exploring the categorical properties of the category $\QCat$ of small $\CQ$-enriched categories and their $\CQ$-functors for a small quantaloid $\CQ$. By showing that $\QCat$ is topological \cite{Adamek1990} over the comma category $\Set/\ob\CQ$ (Proposition \ref{QCat_topological}) one easily describes small limits and colimits in this category, and beyond. In fact, one concludes that categories of this type are total \cite{Street1978} and cototal, hence possess even those limits and colimits of large diagrams whose existence is not made impossible by the size of the small hom-sets of $\QCat$ (see \cite{Borger1990}).

Our greater interest, however, is in the category $\QChu$ whose objects are often called $\CQ$-Chu spaces, the prototypes of which go back to \cite{Barr1991,Pratt1995} and many others (see \cite{Barwise1997,Ganter2007}). Its objects are \emph{$\CQ$-distributors} of $\CQ$-categories (also called $\CQ$-(bi)modules or $\CQ$-profunctors), hence they are compatible $\CQ$-relations (or $\CQ$-matrices) that have been investigated intensively ever since B{\'e}nabou \cite{Benabou1973} introduced them (see \cite{Benabou2000,Borceux1994a}). While when taken as the morphisms of the category whose objects are $\CQ$-categories, they make for a in many ways poorly performing category (as already the case $\CQ={\bf 2}$ shows), when taken as objects of $\QChu$ with morphisms given by so-called \emph{$\CQ$-Chu transforms}, i.e., by pairs of $\CQ$-functors that behave like adjoint operators, we obtain a category that in terms of the existence of limits and colimits behaves as strongly as $\QCat$ itself. In analogy to the property shown in \cite{Giuli2007} in a different categorical context, we first prove that the domain functor $\QChu\to\QCat$ allows for initial liftings \cite{Adamek1990} of structured cones over small diagrams (Theorem \ref{dom_initial_lifting}), which then allows for an explicit description of all limits and colimits in $\QChu$ over small diagrams. But although the domain functor fails to be topological, just as for $\QCat$ we are able to show totality (and, consequently, cototality) of $\QChu$. A key ingredient for this result is the existence proof of a generating set in $\QChu$, and therefore also of a cogenerating set (Theorem \ref{QChu_generator}).

\section{Limits and colimits in $\QCat$}

Throughout, let $\CQ$ be a small \emph{quantaloid}, i.e., a small category enriched in the category $\Sup$ of complete lattices and sup-preserving maps. A small \emph{$\CQ$-category} is given by a set $X$ (its set of objects) and a lax functor
$$a:X\to\CQ,$$
where the set $X$ is regarded as a quantaloid carrying the chaotic structure, so that for all $x,y\in X$ there is precisely one arrow $x\to y$, called $(x,y)$. Explicitly then, the $\CQ$-category structure on $X$ is given by
\begin{itemize}
\item a family of objects $|x|_X:=ax$ in $\CQ$ $(x\in X)$,
\item a family of morphisms $a(x,y):|x|\to|y|$ in $\CQ$ $(x,y\in X)$, subject to
$$1_{|x|}\leq a(x,x)\quad\text{and}\quad a(y,z)\circ a(x,y)\leq a(x,z)$$
\end{itemize}
$(x,y,z\in X)$. When one calls $|x|=|x|_X$ the \emph{extent} (or \emph{type}) of $x\in X$, a \emph{$\CQ$-functor} $f:(X,a)\to(Y,b)$ of $\CQ$-categories $(X,a)$, $(Y,b)$ is an extent-preserving map $f:X\to Y$ such that there is a lax natural transformation $a\to bf$ given by identity morphisms in $\CQ$; explicitly,
$$|x|_X=|f(x)|_Y\quad\text{and}\quad a(x,y)\leq b(f(x),f(y))$$
for all $x,y\in X$. Denoting the resulting ordinary category by $\QCat$, we have a forgetful functor
$$\bfig
\morphism<1200,0>[\QCat`\Set/\ob\CQ;\ob]
\Vtriangle(-300,-500)<300,300>[X`Y`\CQ;f`a`b]
\Vtriangle(900,-500)<300,300>[X`Y`\ob\CQ;f`|\text{-}|`|\text{-}|]
\place(600,-350)[\mapsto] \place(0,-320)[\leq]
\efig$$

\begin{exmp} \phantomsection \label{QCat_exmp}
\begin{itemize}
\item[\rm (1)] If $\CQ$ is a \emph{quantale}, i.e., a one-object quantaloid, then the extent functions of $\CQ$-categories are trivial and $\QCat$ assumes its classical meaning (as in \cite{Kelly1982}, where $\CQ$ is considered as a monoidal (closed) category). Prominent examples are $\CQ={\bf 2}=\{0<1\}$ and $\CQ=([0,\infty],\geq,+)$, where then $\QCat$ is the category $\Ord$ (sets carrying a reflexive and transitive relation, called \emph{order} here but commonly known as preorder, with monotone maps) and, respectively, the category $\Met$ (sets $X$ carrying a distance function $a:X\times X\to[0,\infty]$ required to satisfy only $a(x,x)=0$ and $a(x,z)\leq a(x,y)+a(y,z)$ but, in accordance with the terminology introduced by Lawvere \cite{Lawvere1973} and used in \cite{Hofmann2014}, nevertheless called \emph{metric} here, with non-expanding maps $f:X\to Y$, so that $b(f(x),f(y))\leq a(x,y)$ for all $x,y\in X$).
\item[\rm (2)] (Stubbe \cite{Stubbe2014}) Every quantaloid $\CQ$ gives rise to a new quantaloid $\DQ$ whose objects are the morphisms of $\CQ$, and for morphisms $u,v$ in $\CQ$, a morphism $(u,d,v):u\to/~>/v$ in $\DQ$, normally written just as $d$, is a $\CQ$-morphism $d:\dom u\to\cod v$ satisfying
    $$(d\lda u)\circ u=d=v\circ(v\rda d),$$
    also called a \emph{diagonal} from $u$ to $v$:
    $$\bfig
    \square[\bullet`\bullet`\bullet`\bullet;v\rda d`u`v`d\lda u]
    \morphism(0,500)<500,-500>[\bullet`\bullet;d]
    \efig$$
    (Here $d\lda u$, $v\rda d$ denote the \emph{internal homs} of $\CQ$, determined by
    $$z\leq d\lda u\iff z\circ u\leq d,\quad t\leq v\rda d\iff v\circ t\leq d$$
    for all $z:\cod u\to\cod d$, $t:\dom d\to\dom v$.) With the composition of $d:u\to/~>/v$ with $e:v\to/~>/w$ in $\DQ$ defined by
    $$e\diamond d=(e\lda v)\circ d=e\circ(v\rda d),$$
    and with identity morphisms $u:u\to/~>/u$, $\DQ$ becomes a quantaloid whose local order is inherited from $\CQ$. In fact, there is a full embedding
    $$\CQ\to\DQ,\quad(u:t\to s)\mapsto(u:1_t\to/~>/1_s)$$
    of quantaloids.

    We remark that the construction of $\sD$ works for ordinary categories; indeed it is part of the proper factorization monad on $\CAT$ \cite{Grandis2002}.
\item[\rm (2a)] For $\CQ={\bf 2}$, the quantaloid $\DQ$ has object set $\{0,1\}$. There are exactly two $\DQ$-arrows $1\to/~>/1$, given by $0,1$, and $0$ is the only arrow in every other hom-set of $\DQ$; composition is given by infimum. A $\DQ$-category is given by a set $X$, a distinguished subset $A\subseteq X$ (those elements of $X$ with extent $1$) and a (pre)order on $A$. Hence, a $\DQ$-category structure on $X$ is a (truly!) \emph{partial order} on $X$. With morphisms $f:(X,A)\to(Y,B)$ given by maps $f:X\to Y$ monotone on $A=f^{-1}B$ we obtain the category
    $$\ParOrd=\DQCat,$$
    which contains $\Ord$ as a full coreflective subcategory.
\item[\rm (2b)] For $\CQ=([0,\infty],\geq,+)$, the hom-sets of $\DQ$ are easily described by
    $$\DQ(u,v)=\{s\in[0,\infty]\mid u,v\leq s\},$$
    with composition given by $t\diamond s=s-v+t$ (for $t:v\to/~>/w$). A $\DQ$-category structure on a set $X$ consists of functions $|\text{-}|:X\to[0,\infty]$, $a:X\times X\to[0,\infty]$ satisfying
    $$|x|,|y|\leq a(x,y),\quad a(x,x)\leq|x|,\quad a(x,z)\leq a(x,y)-|y|+a(y,z)$$
    $(x,y,z\in X)$. Obviously, since necessarily $|x|=a(x,x)$, these conditions simplify to
    $$a(x,x)\leq a(x,y),\quad a(x,z)\leq a(x,y)-a(y,y)+a(y,z)$$
    $(x,y,z\in X)$, describing $a$ as a \emph{partial metric} on $X$ (see \cite{Hohle2011,Matthews1994,Pu2012}\footnote{Here our terminology naturally extends Lawvere's notion of metric and is synonymous with ``generalized partial metric'' as used by Pu-Zhang \cite{Pu2012} who dropped finiteness ($a(x,y)<\infty$), symmetry ($a(x,y)=a(y,x)$) and separation ($a(x,x)=a(x,y)=a(y,y)\iff x=y$) from the requirement for the notion of ``partial metric'' as originally introduced by Matthews \cite{Matthews1994}.}). With non-expanding maps $f:(X,a)\to(Y,b)$ satisfying $a(x,x)=b(f(x),f(x))$ for all $x\in X$ one obtains the category
    $$\ParMet=\DQCat,$$
    which contains $\Met$ as a full coreflective subcategory: the coreflector restricts the partial metric $a$ on $X$ to those elements $x\in X$ with $a(x,x)=0$.
\end{itemize}
\end{exmp}

To see how limits and colimits in the (ordinary) category $\QCat$ are to be formed, it is best to first prove its topologicity over $\Set/\ob\CQ$. Recall that, for any functor $U:\CA\to\CX$, a \emph{$U$-structured cone} over a diagram $D:\CJ\to\CA$ is given by an object $X\in\CX$ and a natural transformation $\xi:\De X\to UD$. A \emph{lifting} of $(X,\xi)$ is given by an object $A$ in $\CA$ and a cone $\al:\De A\to D$ over $D$ with $UA=X$, $U\al=\xi$. Such lifting $(A,\al)$ is \emph{$U$-initial} if, for all cones $\be:\De B\to D$ over $D$ and morphisms $t:UB\to UA$ in $\CX$, there is exactly one morphism $h:B\to A$ in $\CA$ with $Uh=t$ and $\al\cdot\De h=\be$. We call $U$ \emph{small-topological} \cite{Giuli2007} if all $U$-structured cones over small diagrams admit $U$-initial liftings, and $U$ is \emph{topological} when this condition holds without the size restriction on diagrams. Recall also the following well-known facts:
\begin{itemize}
\item Topological functors are necessarily faithful \cite{Borger1978}, and for faithful functors it suffices to consider discrete  cones to guarantee topologicity.
\item $U:\CA\to\CX$ is topological if, and only if, $U^{\op}:\CA^{\op}\to\CX^{\op}$ is topological.
\item The two properties above generally fail to hold for small-topological functors. However, for any functor $U$, a $U$-initial lifting of a $U$-structured cone that is a limit cone in $\CX$ gives also a limit cone in $\CA$.
\item Every small-topological functor is a fibration (consider singleton diagrams) and has a fully faithful right adjoint (consider the empty diagram).
\end{itemize}

\begin{prop} \label{QCat_topological}
For every (small) quantaloid $\CQ$, the ``object functor'' $\QCat\to\Set/\ob\CQ$ is topological.
\end{prop}

\begin{proof}
Given a (possibly large) family $f_i:(X,|\text{-}|)\to(Y_i,|\text{-}|_i)$ $(i\in I)$ of maps over $\ob\CQ$, where every $Y_i$ carries a $\CQ$-category structure $b_i$ with extent function $|\text{-}|_i$, we must find a $\CQ$-category structure $a$ on $X$ with extent function $|\text{-}|$ such that (1) every $f_i:(X,a)\to(Y,b)$ is a $\CQ$-functor, and (2) for every $\CQ$-category $(Z,c)$, any extent preserving map $g:Z\to X$ becomes a $\CQ$-functor $(Z,c)\to(X,a)$ whenever all maps $f_i g$ are $\CQ$-functors $(Z,c)\to(Y_i,b_i)$ $(i\in I)$. But this is easy: simply define
$$a(x,y):=\bw_{i\in I}b_i(f_i(x),f_i(y))$$
for all $x,y\in X$. Hence, $a$ is the $\ob$-initial structure on $X$ with respect to the structured source $(f_i:X\to(Y_i,b_i))_{i\in I}$.
\end{proof}

\begin{cor} \label{QCat_complete}
$\QCat$ is complete and cocomplete, and the object functor has both a fully faithful left adjoint and a fully faithful right adjoint.
\end{cor}

\begin{rem} \phantomsection \label{QCat_limit}
\begin{itemize}
\item[\rm (1)] The set of objects of the product $(X,a)$ of a small family of $\CQ$-categories $(X_i,a_i)$ $(i\in I)$ is given by the fibred product of $(X_i,|\text{-}|_i)$ $(i\in I)$, i.e.,
    $$X=\{((x_i)_{i\in I},q)\mid q\in\ob\CQ,\ \forall i\in I(x_i\in X_i,|x_i|=q)\},$$
    and (when writing $(x_i)_{i\in I}$ instead of $((x_i)_{i\in I},q)$ and putting $|(x_i)_{i\in I}|=q$) we have
    $$a((x_i)_{i\in I},(y_i)_{i\in I})=\bw_{i\in I}a_i(x_i,y_i):|(x_i)_{i\in I}|\to|(y_i)_{i\in I}|$$
    for its hom-arrows. In particular, $(\ob\CQ,\top)$ with
    $$\top(q,r)=\top:q\to r$$
    the top element in $\CQ(q,r)$ (for all $q,r\in\ob\CQ$), is the terminal object in $\QCat$.
\item[\rm (2)] The coproduct $(X,a)$ of $\CQ$-categories $(X_i,a_i)$ $(i\in I)$ is simply formed by the coproduct in $\Set$, with all structure to be obtained by restriction:
    $$X=\coprod_{i\in I}X_i,\quad |x|_X=|x|_{X_i}\ \text{if}\ x\in X_i,\quad a(x,y)=\begin{cases}
    a_i(x,y) & \text{if}\ x,y\in X_i,\\
    \bot:|x|\to|y| & \text{else}.
    \end{cases}$$
    In particular, $\varnothing$ with its unique $\CQ$-category structure is an initial object in $\QCat$.
\item[\rm (3)] The equalizer of $\CQ$-functors $f,g:(X,a)\to(Y,b)$ is formed as in $\Set$, by restriction of the structure of $(X,a)$. The object set of their coequalizer $(Z,c)$ in $\QCat$ is also formed as in $\Set$, so that $Z=Y/\sim$, with $\sim$ the least equivalence relation on $Y$ with $f(x)\sim g(x)$, $x\in X$. With $\pi:Y\to Z$ the projection, necessarily $|\pi(y)|_Z=|y|_Y$, and $c(\pi(y),\pi(y'))$ is the join of all
    $$b(y_n,y'_n)\circ b(y_{n-1},y'_{n-1})\circ\dots\circ b(y_2,y'_2)\circ b(y_1,y'_1),$$
    where $|y|=|y_1|,|y'_1|=|y_2|,\dots,|y'_{n-1}|=|y_n|,|y'_n|=|y'|$ $(y_i,y'_i\in Y, n\geq 1)$.
\item[\rm (4)] The fully faithful left adjoint of $\QCat\to\Set/\ob\CQ$ provides a set $(X,|\text{-}|)$ over $\ob\CQ$ with the discrete $\CQ$-structure, given by
    $$a(x,y)=\begin{cases}
    1_{|x|} & \text{if}\ x=y,\\
    \bot:|x|\to|y| & \text{else};
    \end{cases}$$
    while the fully faithful right adjoint always takes $\top:|x|\to|y|$ as the hom-arrow, i.e., it chooses the indiscrete $\CQ$-structure.
\end{itemize}
\end{rem}

\begin{exmp}
The product of partial metric spaces $(X_i,a_i)$ $(i\in I)$ provides its carrier set
$$X=\{((x_i)_{i\in I},s)\mid s\in[0,\infty],\ \forall i\in I(x_i\in X_i,|x_i|=s)\}$$
with the  ``sup metric'':
$$a((x_i)_{i\in I},(y_i)_{i\in I})=\sup_{i\in I}a_i(x_i,y_i).$$
$[0,\infty]$ is terminal in $\ParMet$ when provided with the chaotic metric that makes all distances $0$, and it is a generator when provided with the discrete metric $d$:
$$d(s,t)=\begin{cases}
0 & \text{if}\ s=t,\\
\infty & \text{else}.
\end{cases}$$
\end{exmp}

Beyond small limits and colimits, $\QCat$ actually has all large-indexed limits and colimits that one can reasonably expect to exist. More precisely, recall that an ordinary category $\CC$ with small hom-sets is (see \cite{Borger1990})
\begin{itemize}
\item \emph{hypercomplete} if a diagram $D:\CJ\to\CC$ has a limit in $\CC$ whenever the limit of $\CC(A,D-)$ exists in $\Set$ for all $A\in\ob\CC$; equivalently: whenever, for every $A\in\ob\CC$, the cones $\De A\to D$ in $\CC$ may be labeled by a set;
\item \emph{totally cocomplete} if a diagram $D:\CJ\to\CC$ has a colimit in $\CC$ whenever the colimit of $\CC(A,D-)$ exists in $\Set$ for all $A\in\ob\CC$; equivalently: whenever, for every $A\in\ob\CC$, the connected components of $(A\da D)$ may be labelled by a set.
\end{itemize}
The dual notions are \emph{hypercocomplete} and \emph{totally complete}. It is well known (see \cite{Borger1990}) that
\begin{itemize}
\item $\CC$ is totally cocomplete if, and only if, $\CC$ is \emph{total}, i.e., if the Yoneda embedding $\CC\to\Set^{\CC^{\op}}$ has a left adjoint;
\item total cocompleteness implies hypercompleteness but not vice versa (with Ad{\'a}mek's monadic category over graphs \cite{Adamek1977} providing a counterexample);
\item for a \emph{solid} (=\emph{semi-topological} \cite{Tholen1979}) functor $\CA\to\CX$, if $\CX$ is hypercomplete or totally cocomplete, $\CA$ has the corresponding property \cite{Tholen1980};
\item in particular, every topological functor, every monadic functor over $\Set$, and every full reflective embedding is solid.
\end{itemize}

It is also useful for us to recall \cite[Corollary 3.5]{Borger1990}:

\begin{prop} \label{total_condition}
A cocomplete and cowellpowered category with small hom-sets and a generating set of objects is total.
\end{prop}

Since $\QCat$ is topological over $\Set/\ob\CQ$ which, as a complete, cocomplete, wellpowered and cowellpowered category with a generating and a cogenerating set, is totally complete and totally cocomplete, we conclude:

\begin{thm} \label{QCat_total}
$\QCat$ is totally complete and totally cocomplete and, in particular, hypercocomplete and hypercomplete.
\end{thm}

\begin{rem} \label{QCat_generator}
Of course, we may also apply Proposition \ref{total_condition} directly to obtain Theorem \ref{QCat_total} since the left adjoint of $\QCat\to\Set/\ob\CQ$ sends a generating set of $\Set/\ob\CQ$ to a generating set of $\QCat$, and the right adjoint has the dual property. Explicitly then, denoting for every $s\in\ob\CQ$ by $\{s\}$ the discrete $\CQ$-category whose only object has extent $s$, we obtain the generating set $\{\{s\}\mid s\in\ob\CQ\}$ for $\QCat$. Similarly, providing the disjoint unions $D_s=\{s\}+\ob\CQ$ $(s\in\ob\CQ)$ with the identical extent functions and the indiscrete $\CQ$-category structures, one obtains a cogenerating set in $\QCat$.
\end{rem}

\section{Limits and colimits in $\QChu$}

For $\CQ$-categories $X=(X,a)$, $Y=(Y,b)$, a \emph{$\CQ$-distributor} \cite{Benabou1973} $\phi:X\oto Y$ (also called \emph{$\CQ$-(bi)module} \cite{Lawvere1973}, \emph{$\CQ$-profunctor}) is a family of arrows $\phi(x,y):|x|\to|y|$ $(x\in X,y\in Y)$ in $\CQ$ such that
$$b(y,y')\circ\phi(x,y)\circ a(x',x)\leq\phi(x',y')$$
for all $x,x'\in X$, $y,y'\in Y$. Its composite with $\psi:Y\oto Z$ is given by
$$(\psi\circ\phi)(x,z)=\bv_{y\in Y}\psi(y,z)\circ\phi(x,y).$$
Since the structure $a$ of a $\CQ$-category $(X,a)$ is neutral with respect to this composition, we obtain the category
$$\QDis$$
of $\CQ$-categories and their $\CQ$-distributors which, with the local pointwise order
$$\phi\leq\phi'\iff\forall x,y:\ \phi(x,y)\leq\phi'(x,y),$$
is actually a quantaloid. Every $\CQ$-functor $f:X\to Y$ gives rise to the $\CQ$-distributors $f_{\nat}:X\oto Y$ and $f^{\nat}:Y\oto X$ with
$$f_{\nat}(x,y)=b(f(x),y)\quad\text{and}\quad f^{\nat}(y,x)=b(y,f(x))$$
$(x\in X,y\in Y)$. One has $f_{\nat}\dv f^{\nat}$ in the 2-category $\QDis$, and if one lets $\QCat$ inherit the order of $\QDis$ via
$$f\leq g\iff f^{\nat}\leq g^{\nat}\iff g_{\nat}\leq f_{\nat}\iff 1_{|x|}\leq b(f(x),g(x))\quad (x\in X),$$
then one obtains 2-functors
$$(-)_{\nat}:(\QCat)^{\co}\to\QDis,\quad (-)^{\nat}:(\QCat)^{\op}\to\QDis$$
which map objects identically; here ``$\op$'' refers to the dualization of 1-cells and ``$\co$'' to the dualization of 2-cells.

\begin{exmp}[See Example \ref{QCat_exmp}]
\begin{itemize}
\item[\rm (1)] A ${\bf 2}$-distributor is an \emph{order ideal relation}; that is, a relation $\phi:X\oto Y$ of ordered sets that behaves like a two-sided ideal w.r.t. the order:
$$x'\leq x\ \&\ x\phi y\ \&\ y\leq y'{}\Lra{}x'\phi y'.$$
A $[0,\infty]$-distributor $\phi:X\oto Y$ introduces a distance function between metric spaces $(X,a)$, $(Y,b)$ that must satisfy
$$\phi(x',y')\leq a(x',x)+\phi(x,y)+a(y,y')$$
for all $x,x'\in X$, $y,y'\in Y$.
\item[\rm (2)] A $\sD{\bf 2}$-distributor $\phi:X\oto Y$ is given by a ${\bf 2}$-distributor $A\oto B$ where $A=\{x\in X\mid x\leq x\}$, $B=\{y\in Y\mid y\leq y\}$ are the coreflections of $X$, $Y$, respectively. Likewise, a $\sD[0,\infty]$-distributor $\phi:X\oto Y$ is given by a distributor of the metric coreflections of the partial metric spaces $X$ and $Y$.
\end{itemize}
\end{exmp}

In our context $\QDis$ plays only an auxiliary role for us in setting up the category
$$\QChu$$
whose objects are $\CQ$-distributors and whose morphisms $(f,g):\phi\to\psi$ are given by $\CQ$-functors $f:(X,a)\to(Y,b)$, $g:(Z,c)\to(W,d)$ such that the diagram
\begin{equation} \label{Chu_transform_def_diagram}
\bfig
\square<700,500>[X`Y`W`Z;f_{\nat}`\phi`\psi`g^{\nat}]
\place(350,0)[\circ] \place(350,500)[\circ] \place(0,250)[\circ] \place(700,250)[\circ]
\efig
\end{equation}
commutes in $\QDis$:
\begin{equation} \label{Chu_transform_def}
\psi(f(x),z)=\phi(x,g(z))
\end{equation}
for all $x\in X$, $z\in Z$. In particular, with $\phi=a$, $\psi=b$ one obtains that the morphisms $(f,g):1_{(X,a)}\to 1_{(Y,b)}$ in $\QChu$ are precisely the adjunctions $f\dv g:(Y,b)\to(X,a)$ in the 2-category $\QCat$. With the order inherited from $\QCat$, $\QChu$ is in fact a 2-category, and one has 2-functors
\renewcommand\arraystretch{1.5}
$$\begin{array}{lll}
\dom: & \QChu\to\QCat, & (f,g)\mapsto f,\\
\cod: & \QChu\to(\QCat)^{\op}, & (f,g)\mapsto g.
\end{array}$$
In order for us to exhibit properties of $\QChu$, it is convenient to describe \emph{$\CQ$-Chu transforms}, i.e., morphisms in $\QChu$, alternatively, with the help of \emph{presheaves}, as follows. For every $s\in\ob\CQ$, let $\{s\}$ denote the discrete $\CQ$-category whose only object has extent $s$. For a $\CQ$-category $X=(X,a)$, a \emph{$\CQ$-presheaf $\phi$ on $X$} of extent $|\phi|=s$ is a $\CQ$-distributor $\phi:X\oto\{s\}$. Hence, $\phi$ is given by a family of $\CQ$-morphisms $\phi_x:|x|\to|\phi|$ $(x\in X)$ with $\phi_y\circ a(x,y)\leq\phi_x$ $(x,y\in X)$. With
$$[\phi,\psi]=\bw_{x\in X}\psi_x\lda\phi_x,$$
$\PX$ becomes a $\CQ$-category, and one has the \emph{Yoneda $\CQ$-functor}
$$\sy_X=\sy:X\to\PX,\quad x\mapsto(a(-,x):X\oto\{|x|\}).$$
$\sy$ is fully faithful, i.e., $[\sy(x),\sy(y)]=a(x,y)$ $(x,y\in X)$. The point of the formation of $\PX$ for us is as follows (see \cite{Heymans2010,Shen2015}):

\begin{prop} \label{cograph_Kan_adjunction}
The 2-functor $(-)^{\nat}:(\QCat)^{\op}\to\QDis$ has a left adjoint $\sP$ which maps a $\CQ$-distributor $\phi:X\oto Y$ to the $\CQ$-functor
$$\phi^*:\PY\to\PX,\quad\psi\mapsto\psi\circ\phi;$$
hence,
$$(\phi^*(\psi))_x=\bv_{y\in Y}\psi_y\circ\phi(x,y)$$
for all $\psi\in\PY$, $x\in X$. In particular, for a $\CQ$-functor $f:X\to Y$ one has
$$f^*:=(f_{\nat})^*:\PY\to\PX,\quad (f^*(\psi))_x=\psi_{f(x)}.$$
\end{prop}

Denoting by $\tphi:Y\to\PX$ the transpose of $\phi:X\oto Y$ under the adjunction, determined by $\tphi^{\nat}\circ(\sy_X)_{\nat}=\phi$, so that $(\tphi(y))_x=\phi(x,y)$ for all $x\in X$, $y\in Y$, we can now present $\CQ$-Chu transforms, as follows:

\begin{cor} \label{Chu_transform_presheaf}
A morphism $(f,g):\phi\to\psi$ in $\QChu$ (as in {\rm(\ref{Chu_transform_def_diagram})}) may be equivalently presented as a commutative diagram
\begin{equation}
\bfig
\square/<-`<-`<-`<-/<700,500>[\PX`\PY`W`Z;f^*`\tphi`\tpsi`g]
\efig
\end{equation}
in $\QCat$. Condition {\rm(\ref{Chu_transform_def})} then reads as
\begin{equation}
(\tphi(g(z)))_x=(\tpsi(z))_{f(x)}
\end{equation}
for all $x\in X$, $z\in Z$.
\end{cor}


\begin{proof}
For all $z\in Z$,
$$f^*(\tpsi(z))=\tpsi(z)\circ f_{\nat}=\sy_Z(z)\circ\psi\circ f_{\nat}=\sy_Z(z)\circ g^{\nat}\circ\phi=\sy_Y(g(z))\circ\phi=\tphi(g(z)).$$
\end{proof}

\begin{thm} \label{dom_initial_lifting}
Let $D:\CJ\to\QChu$ be a diagram such that the colimit $W=\colim\cod D$ exists in $\QCat$. Then any cone $\ga:\De X\to\dom D$ in $\QCat$ has a $\dom$-initial lifting $\Ga:\De\phi\to D$ in $\QChu$ with $\phi:X\oto W$, $\dom\Ga=\ga$. In particular, if $\ga$ is a limit cone in $\QCat$, $\Ga$ is a limit cone in $\QChu$.
\end{thm}

\begin{proof}
Considering the functors
$$\QChu\to^{\dom}\QCat\to^{(-)_{\nat}}\QDis,\quad\QChu\to^{\cod}(\QCat)^{\op}\to^{(-)^{\nat}}\QDis,$$
one has the natural transformation
$$\ka:(\dom(-))_{\nat}\oto(\cod(-))^{\nat},\quad\ka_{\phi}:=\phi\ (\phi\in\ob\QChu).$$
By the adjunction of Proposition \ref{cograph_Kan_adjunction}, $\ka D:(\dom D)_{\nat}\oto(\cod D)^{\nat}$ corresponds to a natural transformation $\tkD:\cod D\to\sP(\dom D)_{\nat}$, and the given cone $\ga$ gives a cocone $\ga^*:\sP(\dom D)_{\nat}\to\De\PX$. Forming the colimit cocone $\de:\cod D\to\De W$ one now obtains a unique $\CQ$-functor $\tphi:W\to\PX$ making
$$\bfig
\square/<-`<--`<-`<-/<1000,500>[\De\PX`\sP(\dom D)_{\nat}`\De W`\cod D;\ga^*`\De\tphi`\tkD`\de]
\efig$$
commute in $\QCat$ or, equivalently, making
$$\bfig
\square<1000,500>[\De X`(\dom D)_{\nat}`\De W`(\cod D)^{\nat};\ga_{\nat}`\De\phi`\ka D`\de^{\nat}]
\place(500,0)[\circ] \place(500,500)[\circ] \place(0,250)[\circ] \place(1000,250)[\circ]
\efig$$
commute in $\QDis$, with $\phi:X\oto W$ corresponding to $\widetilde{\phi}$. In other words, we have a cone $\Ga:\De\phi\to D$ with $\dom\phi=X$, $\dom\Ga=\ga$, namely $\Ga=(\ga,\de)$.

Given a cone $\Theta:\De\psi\to D$ with $\psi:Y\oto Z$ in $\QDis$ and a $\CQ$-functor $f:Y\to X$ with $\ga\cdot\De f=\ep:=\dom\Theta$, the cocone $\vartheta:=\cod\Theta:\cod D\to\De Z$ corresponds to a unique $\CQ$-functor $g:W\to Z$ with $\De g\cdot\de=\vartheta$ by the colimit property. As the diagram
$$\bfig
\square|arra|/<-`<-`<-`<-/<1200,600>[\De\PX`\sP(\dom D)_{\nat}`\De W`\cod D;\ga^*``\tkD`\de]
\place(70,350)[\mbox{\scriptsize$\De\tphi$}]
\morphism(0,600)<-500,-500>[\De\PX`\De\PY;]
\place(-300,400)[\mbox{\scriptsize$\De f^*$}]
\morphism(1200,600)|b|<-1700,-500>[\sP(\dom D)_{\nat}`\De\PY;\ep^*]
\morphism<-500,-500>[\De W`\De Z;]
\place(-300,-200)[\mbox{\scriptsize$\De g$}]
\morphism(1200,0)|b|<-1700,-500>[\cod D`\De Z;\vartheta]
\morphism(-500,-500)|l|<0,600>[\De Z`\De\PY;\De\tpsi]
\efig$$
shows, the colimit property of $W$ also guarantees $f^*\tphi=\tpsi g$ (with $\tpsi$ corresponding to $\psi$) which, by Corollary \ref{Chu_transform_presheaf}, means that $(f,g):\psi\to\phi$ is the only morphism in $\QChu$ with $\dom(f,g)=f$ and $\Ga\cdot\De(f,g)=\Theta$.
\end{proof}

\begin{cor}
$\dom:\QChu\to\QCat$ is small-topological; in particular, $\dom$ is a fibration with a fully faithful right adjoint which embeds $\QCat$ into $\QChu$ as a full reflective subcategory. $\cod:\QChu\to(\QCat)^{\op}$ has the dual properties.
\end{cor}

\begin{proof}
With the existence of small colimits guaranteed by Corollary \ref{QCat_complete}, $\dom$-initial liftings to small $\dom$-structured cones exist by Theorem \ref{dom_initial_lifting}. 
For the assertion on $\cod$, first observe that every $\CQ$-category $X=(X,a)$ gives rise to the $\CQ^{\op}$-category $X^{\op}=(X,a^{\circ})$, where $a^{\circ}(x,y)=a(y,x)$ $(x,y\in X)$. With the commutative diagram
$$\bfig
\square<1200,500>[(\QChu)^{\op}`\CQ^{\op}\text{-}\Chu`\QCat`\CQ^{\op}\text{-}\Cat;(-)^{\op}`\cod^{\op}`\dom`(-)^{\op}]
\efig$$
one sees that, up to functorial isomorphisms, $\cod^{\op}:(\QChu)^{\op}\to\QCat$ coincides with the small-topological functor $\dom:\CQ^{\op}\text{-}\Chu\to\CQ^{\op}\text{-}\Cat$.
\end{proof}

\begin{cor} \label{QChu_complete}
$\QChu$ is complete and cocomplete, all small limits and colimits in $\QChu$ are preserved by both $\dom$ and $\cod$.
\end{cor}

\begin{proof}
The $\dom$-initial lifting of a $\dom$-structured limit cone in $\QCat$ is a limit cone in $\QChu$, which is trivially preserved. Having a right adjoint, $\dom$ also preserves all colimits.
\end{proof}

\begin{rem} \phantomsection
\begin{itemize}
\item[\rm (1)] Let us describe (small) products in $\QChu$ explicitly: Given a family of $\CQ$-distributors $\phi_i:X_i\oto Y_i$ $(i\in I)$, one first forms the product $X$ of the $\CQ$-categories $X_i=(X_i,a_i)$ as in Remark \ref{QCat_limit}(1) with projections $p_i$ and the coproduct of the $Y_i=(Y_i,b_i)$ as in Remark \ref{QCat_limit}(2) with injections $s_i$ $(i\in I)$. The transposes $\tphi_i$ then determines a $\CQ$-functor $\tphi$ making the left square of
    $$\bfig
    \square/<-`<-`<-`<-/<700,500>[\PX`\PX_i`Y`Y_i;p_i^*`\tphi`\tphi_i`s_i]
    \square(1500,0)<700,500>[X`X_i`Y`Y_i;(p_i)_{\nat}`\phi`\phi_i`(s_i)^{\nat}]
    \place(1500,250)[\circ] \place(2200,250)[\circ] \place(1850,0)[\circ] \place(1850,500)[\circ]
    \efig$$
    commutative, while the right square exhibits $\phi$ as a product of $(\phi_i)_{i\in I}$ in $\QChu$ with projections $(p_i,s_i)$, $(i\in I)$; explicitly,
    $$\phi(x,y)=(\tphi(y))_x=(\tphi_i(y)\circ(p_i)_{\nat})_x=(\tphi_i(y))_{p_i(x)}=\phi_i(x_i,y)$$
    for $x=((x_i)_{i\in I},q)$ in $X$ and $y=s_i(y)$ in $Y_i$, $i\in I$.
\item[\rm (2)] The coproduct of $\phi_i:X_i\oto Y_i$ $(i\in I)$ in $\QChu$ is formed like the product, except that the roles of domain and codomain need to be interchanged. Hence, one forms the coproduct $X$ of $(X_i)_{i\in I}$ and the product $Y$ of $(Y_i)_{i\in I}$ in $\QCat$ and obtains the coproduct $\phi:X\oto Y$ in $\QChu$ as in
    $$\bfig
    \square<700,500>[X_i`X`Y_i`Y;(s_i)_{\nat}`\phi_i`\phi`(\phi_i)^{\nat}]
    \place(0,250)[\circ] \place(700,250)[\circ] \place(350,0)[\circ] \place(350,500)[\circ]
    \efig$$
    so that $\phi(x,y)=\phi_i(x,y_i)$ for $y=((y_i)_{i\in I},q)$ in $Y$ and $x=s_i(x)$ in $X_i$, $i\in I$.
\item The equalizer of $(f,g),(\of,\og):\phi\to\psi$ in $\QChu$ is obtained by forming the equalizer and coequalizer
    $$U\to^i X\two^f_{\of}Y\quad\text{and}\quad W\two^g_{\og}Z\to^p V$$
    in $\QCat$, respectively. With $\tchi:V\to\PU$ obtained from the coequalizer property making
    $$\bfig
    \square/<-`<-`<-`<-/<700,500>[\PU`\PX`V`W;i^*`\tchi`\tphi`p]
    \efig$$
    commutative, Theorem \ref{dom_initial_lifting} guarantees that
    $$\chi\to^{(i,p)}\phi\two^{(f,g)}_{(\of,\og)}\psi$$
    is an equalizer diagram in $\QChu$, where
    $$\chi(x,p(w))=(\tchi(p(w))_x=(i^*(\tphi(w)))_x=(\tphi(w)\circ i_{\nat})_x=\phi(i(x),w)$$
    for all $x\in U$, $w\in W$.
\item Coequalizers in $\QChu$ are formed like equalizers, except that the roles of domain and codomain need to be interchanged.
\end{itemize}
\end{rem}

We will now strengthen Corollary \ref{QChu_complete} and show total completeness and total cocompleteness of $\QChu$ with the help of Proposition \ref{total_condition}. To that end, let us observe that, since the limit and colimit preserving functors $\dom$ and $\cod$ must in particular preserve both monomorphisms and epimorphisms, a monomorphism $(f,g):\phi\to\psi$ in $\QChu$ must be given by a monomorphism $f$ and an epimorphism $g$ in $\QCat$, i.e., by an injective $\CQ$-functor $f$ and a surjective $\CQ$-functor $g$. Consequently, $\QChu$ is wellpowered, and so is its dual $(\QChu)^{\op}\cong\CQ^{\op}\text{-}\Chu$. The main point is therefore for us to prove:

\begin{thm} \label{QChu_generator}
$\QChu$ contains a generating set of objects and, consequently, also a cogenerating set.
\end{thm}

\begin{proof}
With the notations explained below, we show that
$$\{\eta_s:\varnothing\oto D_s\mid s\in\ob\CQ\}\cup\{\lam_t:\{t\}\oto\hC\mid t\in\ob\CQ\}$$
is generating in $\QChu$. Here $D_s$ belongs to a generating set of $\QCat$ (see Remark \ref{QCat_generator}), and
$$C=\coprod\limits_{t\in\ob\CQ}\sP\{t\}$$
is a coproduct in $\QCat$ (see Remark \ref{QCat_limit}(2)) of the presheaf $\CQ$-categories of the singleton $\CQ$-categories $\{t\}$ (see Proposition \ref{cograph_Kan_adjunction}). From $C$ one obtains $\hC$ by adding an isomorphic copy of each object in $C$, which may be easily explained for a $\CQ$-category $(X,a)$: simply provide the set $\hX:=X\times\{1,2\}$ with the structure
$$|(x,i)|_{\hX}=|x|_{X}\quad\text{and}\quad\ha((x,i),(y,j))=a(x,y)$$
for all $x,y\in X$, $i,j\in\{1,2\}$. Noting that the objects of $\sP\{t\}$ are simply $\CQ$-arrows with domain $t$, we now define $\lam_t:\{t\}\oto\hC$ by
$$\lam_t(u,i)=\begin{cases}
u & \text{if}\ \dom u=t,\\
\bot & \text{else}
\end{cases}$$
for $i\in\{1,2\}$ and every object $t$ and arrow $u$ in $\CQ$. For another element $(v,j)$ in $\hC$, if $\dom v=\dom u=t$ one then has
$$[(u,i),(v,j)]\circ\lam_t(u,i)=(v\lda u)\circ u\leq v=\lam_t(v,j),$$
and in other cases this inequality holds trivially. Hence, $\lam_t$ is indeed a $\CQ$-distributor.

Let us now consider $\CQ$-Chu transforms $(f,g)\neq(\of,\og):\phi\to\psi$ as in
$$\bfig
\square|alra|/@{->}@<3pt>`->`->`@{->}@<3pt>/<800,500>[(X,a)`(Y,b)`(W,d)`(Z,c);f_{\nat}`\phi`\psi`g^{\nat}]
\morphism(0,500)|b|/@{->}@<-3pt>/<800,0>[(X,a)`(Y,b);\of_{\nat}]
\morphism(0,0)|b|/@{->}@<-3pt>/<800,0>[(W,d)`(Z,c);\og^{\nat}]
\place(0,250)[\circ] \place(800,250)[\circ] \place(400,-30)[\circ] \place(400,30)[\circ] \place(400,470)[\circ] \place(400,530)[\circ]
\efig$$

{\bf Case 1}: $X=\varnothing$ is the initial object of $\QCat$ (and $\QDis$). Then $g\neq\og$, and we find $s\in\ob\CQ$ and $h:W\to D_s$ with $hg\neq h\og$ in $\QCat$. Consequently, $(1_{\varnothing},h):\eta_s\to\phi$ satisfies $(f,g)(1_{\varnothing},h)\neq(\of,\og)(1_{\varnothing},h)$.

{\bf Case 2}: $f\neq\of$, so that $f(x_0)\neq\of(x_0)$ for some $x_0\in X$. Then, for $t:=|x_0|$, $e:\{t\}\to X$, $|x_0|\mapsto x_0$, is a $\CQ$-functor with $fe\neq\of e$, and it suffices to show that
$$h:W\to\hC,\quad w\mapsto(\phi(x_0,w),1)$$
is a $\CQ$-functor making $(e,h):\lam_t\to\phi$ a $\CQ$-Chu transform. Indeed,
\begin{align*}
&d(w,w')\leq\phi(x_0,w')\lda\phi(x_0,w)=[h(w),h(w')],\\
&\lam_t(h(w))=\phi(x_0,w)=\phi(e(t),w)
\end{align*}
for all $w,w'\in W$.

{\bf Case 3}: $X\neq\varnothing$ and $g\neq\og$. Then $g(z_0)\neq\og(z_0)$ for some $z_0\in Z$, and with any fixed $x_0\in X$ we may alter the previous definition of $h:W\to\hC$ by
$$h(w):=\begin{cases}
(\phi(x_0,w),2) & \text{if}\ w=\og(z_0),\\
(\phi(x_0,w),1) & \text{else}.
\end{cases}$$
The verification for $h$ to be a $\CQ$-functor and $(e,h):\lam_t\to\phi$ a $\CQ$-Chu transform remain intact, and since $hg\neq h\og$, the proof is complete.
\end{proof}

\begin{rem}
A generating set in $\QChu$ may be alternatively given by
$$\{\lam_{\varnothing}:\varnothing\oto\hC\}\cup\{\lam_t:\{t\}\oto\hC\mid t\in\ob\CQ\},$$
so that in Case 1 one may proceed exactly as in Case 3 only by replacing $\phi(x_0,w)$ with $\top:q\to|w|$ for any fixed $q\in\ob\CQ$.
\end{rem}

With Theorem \ref{QChu_generator} we obtain:

\begin{cor}
$\QChu$ is totally complete and totally cocomplete and, in particular, hypercocomplete and hypercomplete.
\end{cor}

\end{document}